\documentclass[12pt]{amsart}
\usepackage[margin=1in]{geometry}
\usepackage{amsmath, amsthm, amssymb}
\usepackage{enumerate}
\usepackage[centertableaux]{ytableau}
\usepackage{todonotes}

\newtheorem{thm}{Theorem}[section]

\newtheorem{prop}[thm]{Proposition}

\newtheorem{lemma}[thm]{Lemma}

\theoremstyle{definition}

\newtheorem*{defn}{Definition}

\theoremstyle{remark}
\newtheorem{rmk}[thm]{Remark}

\newcommand{\by}{\begin{ytableau}}
\newcommand{\ey}{\end{ytableau}}
\newcommand{\yts}{\ytableaushort}

\renewcommand{\bar}{\overline}

\newcommand{\cal}{\mathcal}

\DeclareMathOperator{\CYT}{CYT}

\begin{document}

\title{A simplified Kronecker rule for one hook shape}
\author{Ricky Ini Liu}
\address{Department of Mathematics, North Carolina State University, Raleigh, NC 27695}
\email{riliu@ncsu.edu}

\begin{abstract}
	Recently Blasiak gave a combinatorial rule for the Kronecker coefficient $g_{\lambda \mu \nu}$ when $\mu$ is a hook shape by defining a set of colored Yamanouchi tableaux with cardinality $g_{\lambda\mu\nu}$ in terms of a process called conversion. We give a characterization of colored Yamanouchi tableaux that does not rely on conversion, which leads to a simpler formulation and proof of the Kronecker rule for one hook shape.
\end{abstract}
\maketitle

\section{Introduction}

Let $V_\lambda$ be the irreducible representation of the symmetric group $S_n$ corresponding to the partition $\lambda$. The \emph{Kronecker coefficient} $g_{\lambda\mu\nu}$ is defined to be the multiplicity of $V_\nu$ in the tensor product $V_\lambda \otimes V_\mu$. Despite their simple definition and fundamental importance, the Kronecker coefficients are not well understood: indeed, a longstanding open problem is to give a positive combinatorial formula for $g_{\lambda\mu\nu}$. For some known special cases, see \cite{Ballantine-Orellana,Blasiak,Blasiak-Mulmuley-Sohoni,Briand-Orellana-Rosas,Lascoux, Remmel, Remmel-Whitehead, Rosas}.

In \cite{Blasiak}, Blasiak gives a combinatorial rule for $g_{\lambda\mu\nu}$ when only one of the partitions is a hook shape, which represents the first nontrivial Kronecker rule when two of the partitions are completely unrestricted. This rule describes the Kronecker coefficients as the cardinalities of certain sets of \emph{colored Yamanouchi tableaux}. The characterization of these tableaux as well as the proof that they are correctly enumerated relies on two processes introduced by Haiman \cite{Haiman} called \emph{mixed insertion} and \emph{conversion}. As a result, it is unfortunately somewhat complicated to describe and apply this Kronecker rule in practice.

In this paper, we give an alternative characterization of the colored Yamanouchi tableaux from \cite{Blasiak} that does not make reference to mixed insertion or conversion. The result is a simplified formulation and proof of the Kronecker rule for one hook shape. While we use much of the same notation as \cite{Blasiak} for reasons of consistency, the proof contained herein is essentially self-contained.

We begin in Section 2 with some preliminaries about colored tableaux and conversion. In Section 3 we discuss reading words of colored tableaux and prove our main lemma (Lemma~\ref{lemma-main}). Finally, in Section 4 we show how Lemma~\ref{lemma-main} implies the Kronecker rule for one hook shape, given in Theorem~\ref{thm-main}.

\section{Preliminaries}
In this section, we review the necessary background on colored tableaux and conversion.

\subsection{Young diagrams}
Given a positive integer $n$, a \emph{partition} $\lambda=(\lambda_1, \lambda_2, \dots) \vdash n$ is a weakly decreasing sequence of nonnegative integers summing to $n$. (We often ignore trailing zeroes.) The \emph{Young diagram} of a partition $\lambda$ is an array of boxes, aligned to the north and west, with $\lambda_i$ boxes in row $i$. We typically refer to a partition and its Young diagram interchangeably. Transposing the Young diagram of $\lambda$ gives the Young diagram of its \emph{conjugate} $\lambda'$. Given two partitions $\lambda$ and $\mu$ with $\mu_i \leq \lambda_i$ for all $i$, the \emph{skew Young diagram} $\lambda/\mu$ is the set of boxes in $\lambda$ with the boxes of $\mu$ removed.

A \emph{ribbon} is a skew Young diagram that does not contain a $2 \times 2$ square. Any ribbon can be divided into \emph{connected components}, where boxes $x$ and $y$ lie in the same component if there is a sequence of boxes $x=x_1, x_2, \dots, x_k=y$ in the ribbon such that $x_i$ and $x_{i+1}$ share an edge.

\subsection{Colored tableaux}
Let $\cal A$ be the alphabet $\{1, 2, \dots, n, \bar 1, \bar 2, \dots, \bar n\}$ endowed with the partial order $1<2<\dots<n$ and $\bar 1<\bar 2 < \dots < \bar n$. Note that there are $\binom{2n}{n}$ total orders of $\cal A$ compatible with this partial order.

For any total order $<$ of $\cal A$, a \emph{(semistandard) colored tableau} $T = T_<$ of shape $\lambda$ is a filling of the boxes of $\lambda$ with letters of $\cal A$ such that: 
\begin{itemize}
	\item each row and column is weakly increasing (with respect to $<$);
	\item no two identical unbarred letters appear in the same column; and
	\item no two identical barred letters appear in the same row.
\end{itemize}
See Figure~\ref{fig-conversion} for some examples of colored tableaux with respect to different orders.

Observe that if a colored tableau contains only barred letters, then transposing the tableau and removing all bars gives a semistandard tableau containing only unbarred letters.

The \emph{unbarred content} of a colored tableau $T_<$ is the sequence $\alpha(T_<) = (\alpha_1, \alpha_2, \dots)$, where $\alpha_i$ is the number of occurrences of $i$ in $T_<$, while the \emph{(total) content} is $(c_1, c_2, \dots)$, where $c_i$ is the total number of occurrences of $i$ and $\bar i$ in $T_<$. The \emph{total color} is the number of barred letters in $T_<$.

If $a$ and $\bar b$ are consecutive letters in $<$ (in either order), then for any tableau $T_<$, the letters $a$ and $\bar b$ must occupy boxes forming a ribbon. Any connected component of a ribbon can be filled by $a$ and $\bar b$ in exactly two ways: these differ only in the northeast corner if $a<\bar b$ or in the southwest corner if $a>\bar b$.

\subsection{Conversion}
Suppose $<$ and $\prec$ are total orders of $\cal A$ that are identical except for the order of $a$ and $\bar b$, say $a<\bar b$ but $\bar b \prec a$. Then there is a natural bijection between colored tableaux with respect to $<$ and colored tableaux with respect to $\prec$ via a process called \emph{conversion}, introduced by Haiman \cite{Haiman} (see also \cite{Benkart-Sottile-Stroomer}).

Let $T_<$ be a colored tableau, so that the boxes containing $a$ or $\bar b$ form a ribbon. For each component of the ribbon, there is a unique way to refill it with the same number of $a$'s and $\bar b$'s in a way that is compatible with $\bar b\prec a$. (If the northeast corner contains $a$, then move each $a$ to the bottom of its column; if the northeast box contains $\bar b$, move each $a$ to the right within its row. See Figure~\ref{fig-ribbon}.) Replacing each component in this manner switches $T_<$ to a colored tableau $T_\prec$.

\begin{figure}
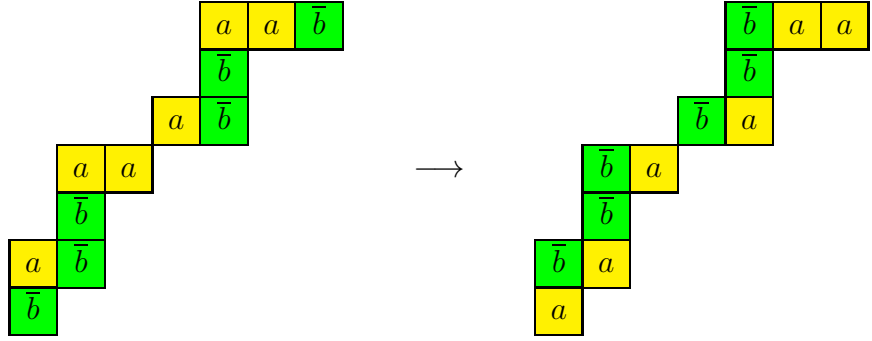
 
	\[
	\yts{\none\none\none\none {*(yellow)a}{*(yellow)a}{*(green)\bar b},
		\none\none\none\none {*(green)\bar b},
		\none\none\none {*(yellow)a} {*(green)\bar b},
		\none {*(yellow)a}{*(yellow)a},
		\none {*(green)\bar b},
		{*(yellow)a}{*(green)\bar b},
		{*(green)\bar b}}\qquad \longrightarrow \qquad
	\yts{\none\none\none\none {*(green)\bar b}{*(yellow)a}{*(yellow)a},
		\none\none\none\none {*(green)\bar b},
		\none\none\none {*(green)\bar b}{*(yellow)a},
		\none {*(green)\bar b}{*(yellow)a},
		\none {*(green)\bar b},
		{*(green)\bar b}{*(yellow)a},
		{*(yellow)a}}
	\]
	\caption{\label{fig-ribbon}Conversion from $a < \bar b$ to $\bar b \prec a$. The ribbon above has two connected components with six boxes each.}
\end{figure}

Given any two total orders $<$ and $\prec$ of $\mathcal A$, one can be obtained from the other by repeatedly switching the order of a consecutive barred letter and unbarred letter. Therefore, one can iterate the switching process above to convert any tableau $T_<$ to a tableau $T_\prec$. See Figure~\ref{fig-conversion} for an example. Importantly, the resulting tableau is well-defined, that is, it does not depend on the sequence of switches used to transform $<$ to $\prec$ (as follows from the Diamond Lemma, or see \cite{Benkart-Sottile-Stroomer}).
\begin{figure}
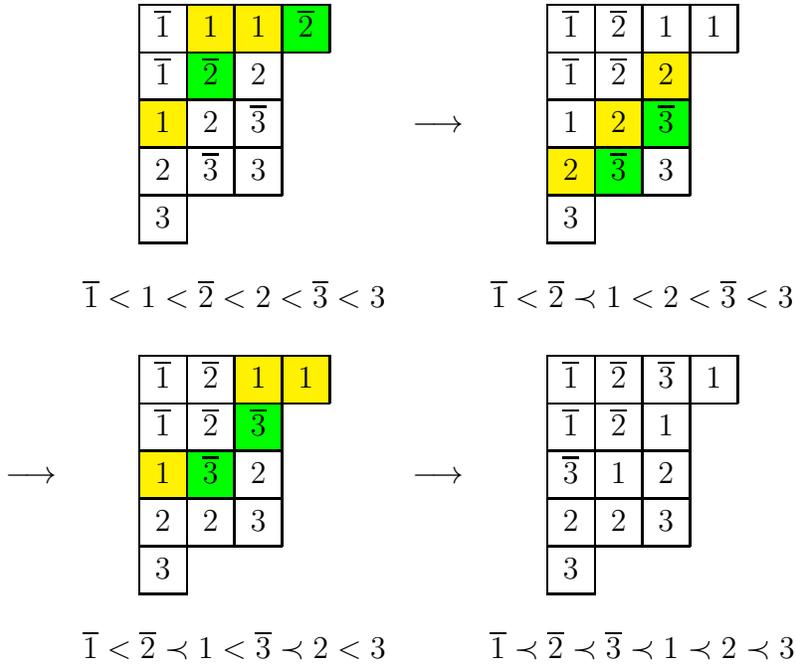
 
	\[\begin{array}{cccc}
	&
	\yts{{\bar 1}{*(yellow)1}{*(yellow)1}{*(green)\bar 2},{\bar1}{*(green)\bar2}2,{*(yellow)1}2{\bar3},2{\bar3}3,3}
	&
	\longrightarrow
	&
	\yts{{\bar 1}{\bar 2}11,{\bar1}{\bar2}{*(yellow)2},1{*(yellow)2}{*(green)\bar3},{*(yellow)2}{*(green)\bar3}3,3}
	\\\\
	&
	\bar 1 < 1 < \bar 2 < 2 < \bar 3 < 3
	&
	&
	\bar 1 < \bar 2 \prec 1 < 2 < \bar 3 < 3
	\\\\
	\longrightarrow&
	\yts{{\bar 1}{\bar 2}{*(yellow)1}{*(yellow)1},{\bar1}{\bar2}{*(green)\bar 3},{*(yellow)1}{*(green)\bar3}2,223,3}
	&
	\longrightarrow
	&
	\yts{{\bar 1}{\bar 2}{\bar 3}1,{\bar1}{\bar2}1,{\bar3}12,223,3}
	\\\\
	&
	\bar 1 < \bar 2 \prec 1 <  \bar 3 \prec 2 < 3
	&&
	\bar 1 \prec  \bar 2 \prec  \bar 3 \prec 1\prec 2\prec 3
	\end{array}\]
	\caption{\label{fig-conversion} Conversion from the natural order $\bar 1 < 1 < \bar 2 < 2 < \bar 3 < 3$ to the small bar order $\bar 1 \prec \bar 2 \prec \bar 3 \prec 1 \prec 2 \prec 3$ in three steps.}
\end{figure}

\section{Reading words}

In this section, we define a particular reading word of a colored tableau and show how it is affected by conversion.

\subsection{Ballot sequences} 
	Given a word $w=w_1 \dotsm w_m$ in the alphabet $\{1, \dots, n\}$, we say that $w$ is a \emph{ballot sequence} if any initial segment of $w$ contains at least as many occurrences of $i$ as of $i+1$ for each letter $i$. (Ballot sequences are also called \emph{lattice words} or \emph{reverse Yamanouchi words}.) For instance, $11232132$ is a ballot sequence, but $11232213$ is not because the initial segment $112322$ contains more $2$'s than $1$'s.

\begin{defn}
	For a sequence $\alpha = (\alpha_1, \dots, \alpha_n)$, we say that $w$ is an $\alpha$-ballot sequence if, in any initial segment of $w$, the number of occurrences of $i+1$ minus the number of occurrences of $i$ is at most $\alpha_{i}-\alpha_{i+1}$. 
\end{defn}
When $\alpha$ is a partition, $w$ is an $\alpha$-ballot sequence if and only if the concatenation \[\underbrace{11\cdots 1}_{\alpha_1} \underbrace{22\cdots 2}_{\alpha_2} \cdots \underbrace{nn \cdots n}_{\alpha_n} w\] is a ballot sequence.

\subsection{Reading words}
We specify a particular reading order for colored tableaux as follows.
\begin{defn}
	Let $T_<$ be a colored tableau (with respect to any total order $<$ of $\cal A$).
	\begin{itemize}
		\item The \emph{unbarred reverse (row) reading word} $u=u(T_<)$ of $T_<$ is the word obtained by reading its unbarred entries along rows from right to left, top to bottom, skipping over any barred entries.
		\item The \emph{barred (column) reading word} $\bar v = \bar v(T_<)$ of $T_<$ is the word obtained by reading the barred entries along columns from bottom to top, left to right, skipping over any unbarred entries. Let $v=v(T_<)$ to be the word obtained from $\bar v$ by removing the bar from each letter.
		\item The \emph{total reverse (row-column) reading word} $w=w(T_<)$ is the concatenation $w=uv$.
	\end{itemize}
\end{defn}

For example, if $T_<$ is the first tableau in Figure~\ref{fig-conversion}, then $u(T_<)=11221323$, $v(T_<)=113232$, and $w(T_<) = 11221323113232$ (which is a ballot sequence).

Two types of total orders will be especially important for determining when the property of $w$ being a ballot sequence is preserved under conversion.

\begin{defn}
	Let $<$ be a total order on $\cal A$.
	\begin{itemize}
		\item The order $<$ is \emph{unbarred-tight} if at most one barred letter occurs between any two consecutive unbarred letters.
		\item The order $<$ is \emph{barred-tight} if at most one unbarred letter occurs between any two consecutive barred letters.
	\end{itemize}
\end{defn}
For instance, the order $1<\bar 1 < \bar 2 < 2 < \bar 3 < 3$ is barred-tight but not unbarred-tight. Most total orders are neither barred-tight nor unbarred-tight, and few are both. However, among the orders that are both barred-tight and unbarred-tight are the \emph{small bar order}
\[\bar 1 \prec \bar 2 \prec \dots \prec \bar n \prec 1 \prec 2 \prec \dots \prec n\]
and the \emph{natural order}
\[\bar 1 < 1 < \bar 2 < 2 < \dots < \bar n < n.\]

We are now ready to prove the main lemma needed to derive the hook Kronecker rule.
\begin{lemma} \label{lemma-main}
	Let $<$ and $\prec$ be total orders of $\cal A$, and let $T_<$ and $T_\prec$ be colored tableaux that correspond under conversion with unbarred content $\alpha = \alpha(T_<) = \alpha(T_\prec)$.
	\begin{enumerate}[(a)]
		\item If $<$ and $\prec$ are both unbarred-tight, then $u(T_<)$ is a ballot sequence if and only if $u(T_\prec)$ is. 
		\item If $<$ and $\prec$ are both barred-tight, then $v(T_<)$ is an $\alpha$-ballot sequence if and only if $v(T_\prec)$ is.
		\item If $<$ and $\prec$ are both unbarred-tight and barred-tight, then $w(T_<)$ is a ballot sequence if and only if $w(T_\prec)$ is.
	\end{enumerate}
\end{lemma}
\begin{proof}
	Part (c) follows immediately from parts (a) and (b): by definition, $w$ is a ballot sequence if and only if $u$ is a ballot sequence and $v$ is an $\alpha$-ballot sequence.	We prove part (a) below; part (b) is similar.
	
	Let $<$ be any unbarred-tight order, and find the smallest pair of consecutive letters of the form $a < \bar b$. If $a\neq 1$, then the letter preceding $a$ must be $a-1$, and if $a \neq n$, the letter succeeding $\bar b$ must be $a+1$. Then define $\prec$ to be the unbarred-tight order given by the switch
	\[ \cdots < a-1 < a < \bar b < a+1 < \cdots  \quad \to \quad \cdots \prec a-1 \prec \bar b \prec a \prec a+1 \prec \cdots.\]
	Switching from $T_<$ to $T_\prec$ only shifts the occurrences of $a$ within $u(T_<)$ and $u(T_\prec)$.
	
	Denote the $i$th occurrence of $a$ in $u(T_<)$ or $u(T_\prec)$ by $a_i$ (and similarly for other letters). Note that for either $T_<$ or $T_\prec$, if $a_i$ is the rightmost occurrence of $a$ or $\bar b$ in its row, then the occurrences of $a+1$ that appear before $a_i$ in $u$ are exactly those that lie in a column to the right of $a_i$. (Since $a+1$ immediately follows $a$ and $\bar b$ in $<$ and $\prec$, any label directly east of $a_i$ must be at least $a+1$, so no occurrence of $a+1$ can be strictly southeast of $a_i$.) Similarly, if $a_i$ is the leftmost occurrence of $a$ or $\bar b$ in its row, then the occurrences of $a-1$ that appear before $a_i$ in $u$ are exactly those that lie in a column weakly to the right of $a_i$.
	
	Suppose $u(T_<)$ is a ballot sequence. When switching from $T_<$ to $T_\prec$, each $a$ either moves to the right within its row past some $\bar b$'s, which leaves it in the same position in $u(T_\prec)$ as in $u(T_<)$, or it moves down within its column, which shifts it later in $u(T_\prec)$ than in $u(T_<)$. Hence we need only verify that $a_i$ occurs before $(a+1)_i$ in $u(T_\prec)$ when $a_i$ gets shifted down within its column $c_i$. In this case, $a_i$ is the rightmost $a$ or $\bar b$ in its row in $T_\prec$, so we need at most $i-1$ occurrences of $a+1$ to the right of column $c_i$ in $T_\prec$ or, equivalently, in $T_<$.
	
	In the ribbon containing $a$ and $\bar b$ in $T_<$, consider the connected component containing $a_i$. Since $a_i$ gets shifted down, the northeast corner of this component must contain some $a_j$, $j \leq i$, lying in column $c_j = c_i+i-j$. Since $a_j$ is the rightmost $a$ or $\bar b$ in its row and $u(T_<)$ is a ballot sequence, there are at most $j-1$ occurrences of $a+1$ to the right of column $c_j$. But since each column can contain at most one occurrence of $a+1$, there can be at most $j-1+(c_j-c_i) = i-1$ occurrences of $a+1$ to the right of column $c_i$, as desired.
	
	Conversely, if $u(T_\prec)$ is a ballot sequence, then we need only verify that $a_i$ occurs after $(a-1)_i$ in $T_<$ when $a_i$ is shifted within its column. This follows from a similar argument to above.
	
	Therefore $u(T_<)$ is a ballot sequence if and only if $u(T_\prec)$ is whenever $<$ and $\prec$ are related by the single switch above. But repeatedly applying such switches will transform any unbarred-tight order to the small bar order. It follows that if $u(T_<)$ is a ballot sequence for any unbarred-tight order, then $u(T_\prec)$ must be for any unbarred-tight order $\prec$, as desired.
\end{proof}

\begin{rmk}
	The proof of Lemma~\ref{lemma-main} in fact shows slightly more: if $u(T_<)$ is a ballot sequence for some unbarred-tight order $<$, then $u(T_\prec)$ is a ballot sequence for any total order $\prec$. A similar result holds for $v$ if $<$ is barred-tight and for $w$ if $<$ is both barred-tight and unbarred-tight. However, it is not true that the ballot sequence condition is always preserved by conversion: for instance, the two tableaux $T_<$ and $T_\prec$ below are related by conversion, but $w(T_\prec) = 1212$ is a ballot sequence while $w(T_<) = 2112$ is not.
	\[
	\begin{array}{ccc}
	\yts{1{\bar 2},{\bar1}2}& \longleftrightarrow &  \yts{12,{\bar1}{\bar2}}\\\\
	1 \prec \bar 1 \prec \bar 2 \prec 2 && 1 < \bar 1 < 2 < \bar 2
	\end{array}
	\]
\end{rmk}

\begin{rmk}
	Although we present Lemma~\ref{lemma-main} using a specific reading order for simplicity, the same result holds for many other reading orders. In particular, let $u'(T_<)$ be any reading word in which a box containing $a$ is read after any box containing $a$, $a-1$, or $a+1$ that lies weakly to the northeast or immediately to the northwest. It is easy to show that if $<$ is an unbarred-tight order, then $u(T_<)$ is a ballot sequence if and only if $u'(T_<)$ is.
	
	Similarly, for any barred-tight order, $v$ can be replaced with any reading order in which $\bar b$ is read after any box containing $\bar b$, $\bar{b-1}$, or $\bar{b+1}$ that lies weakly to the southwest or immediately to the northwest.
\end{rmk}

\begin{rmk}
	Conversion does not in general preserve the Knuth equivalence class of the reverse of $w$ even when $<$ and $\prec$ are the natural order and the small bar order: take
	\[T_< = \yts{{\bar 1}1,2{\bar3},3}\,, \qquad T_\prec = \yts{{\bar 1}1,{\bar 3}2, 3}\,.\]
	
\end{rmk}

\section{Kronecker coefficients} We are now ready to prove a combinatorial rule for Kronecker coefficients with one hook shape.  For $0 \leq d \leq n-1$, we will denote by $\mu(d)$ the partition $(n-d, \underbrace{1, \dots, 1}_d) = (n-d, 1^d)$. 

First we recall the following well-known result relating colored tableaux to the sum of two Kronecker coefficients. This follows from the classical theory of Littlewood-Richardson coefficients; we sketch the proof below (see also Proposition 3.1 of \cite{Blasiak}).
\begin{prop} \label{prop-sum}
	Let $\lambda$ and $\nu$ be partitions of $n$, and let $\prec$ be the small bar order on $\mathcal A$ given by
	\[\bar 1 \prec \bar 2 \prec \dots \prec \bar n \prec 1 \prec 2 \prec \dots \prec n.\] Then the number of colored tableaux $T_\prec$ of content $\lambda$, total color $d$, and shape $\nu$ such that the total reverse reading word $w(T_\prec)$ is a ballot sequence is
	$g_{\lambda \mu(d) \nu} + g_{\lambda \mu(d-1) \nu}$.
\end{prop}
\begin{proof}
	 Let $\langle \cdot, \cdot \rangle$ be the Hall inner product on the ring of symmetric functions, so that, for instance, $\langle s_\lambda, s_\mu s_\nu \rangle$ is the Littlewood-Richardson coefficient $c^\lambda_{\mu\nu}$. Let $*$ denote the internal product on symmetric functions defined by $\langle s_\lambda, s_\mu * s_\nu \rangle = g_{\lambda\mu\nu}$.
	 
	 Then
	 \begin{align*}
	 g_{\lambda\mu(d)\nu} + g_{\lambda\mu(d-1)\nu} &= \langle s_\lambda,  s_{\mu(d)}*s_\nu \rangle + \langle s_\lambda, s_{\mu(d-1)}*s_\nu\rangle\\
	 &= \langle s_\lambda, (s_{\mu(d)}+s_{\mu(d-1)})* s_\nu \rangle\\
	 &= \langle s_\lambda, (s_{(n-d)}\cdot s_{(1^d)})* s_\nu \rangle\\
	 &= \sum_{\beta \vdash d} \langle s_\lambda, (s_{(n-d)}*s_{\nu/\beta})(s_{(1^d)}*s_{\beta})\rangle\\
	 &= \sum_{\beta \vdash d} \langle s_\lambda, s_{\nu/\beta} s_{\beta'}\rangle.
	 \end{align*}
	 Let $\nu/\beta \oplus \beta'$ be the skew Young diagram consisting of the disjoint union of $\nu/\beta$ and $\beta'$, translated to lie in distinct rows and columns (with $\beta'$ to the southwest). By the Littlewood-Richardson rule, $\langle s_\lambda, s_{\nu/\beta} s_{\beta'} \rangle$ is the number of semistandard tableaux $T$ of shape $\nu/\beta \oplus \beta'$ and content $\lambda$ whose reverse (row) reading word is a ballot sequence. By transposing the part of $T$ that lies in $\beta'$, barring each letter within, and combining with the part of $T$ lying in $\nu/\beta$, we arrive at a colored tableau $T_\prec$ of content $\lambda$, total color $d$ (whose barred letters form the shape $\beta$), and shape $\nu$ such that $w(T_\prec)$ is a ballot sequence. Summing over all $\beta$ gives the desired result.
\end{proof}

We are now ready to prove a reformulated version of Hook Kronecker Rule I from \cite{Blasiak}.
\begin{thm} \label{thm-main}
	Let $\lambda$ and $\nu$ be partitions of $n$, and let $<$ be the natural order on $\cal A$ given by
	\[\bar 1 <1< \bar 2 < 2 < \dots < \bar n < n.\]
	
	Then the Kronecker coefficient $g_{\lambda \mu(d) \nu}$ is the number of colored tableaux $T_<$ of content $\lambda$, total color $d$, and shape $\nu$ such that the total reverse (row-column) reading word $w(T_<)$ is a ballot sequence and the southwest corner of $T_<$ is unbarred.
\end{thm}
\begin{proof}
	By Lemma~\ref{lemma-main}(c) and Proposition~\ref{prop-sum}, the number of colored tableaux $T_<$ of content $\lambda$, total color $d$, and shape $\nu$ such that $w(T_<)$ is a ballot sequence is $g_{\lambda\mu(d)\nu}+g_{\lambda\mu(d-1)\nu}$. Denote this set by $\CYT_{\lambda, d}(\nu)$, and write $\CYT^+_{\lambda, d}(\nu)$ (resp. $\CYT^-_{\lambda, d}(\nu)$) for the subset of these tableaux with barred (resp. unbarred) southwest corner. 
	
	Since $i$ immediately follows $\bar i$ in $<$ for all unbarred letters $i \in \cal A$, the southwest corner of $T_<$ can always be toggled from barred to unbarred or vice versa and remain a colored tableau. Moreover, this toggle does not change $w(T_<)$; it simply moves the last letter of $u(T_<)$ to be the first letter of $v(T_<)$ or vice versa. Hence $|\CYT^+_{\lambda, d}(\nu)| = |\CYT^-_{\lambda, d-1}(\nu)|$.
	
	The result now follows easily, say, by induction on $d$: if $|\CYT^-_{\lambda, d-1}(\nu)| = g_{\lambda\mu(d-1)\nu}$, then \[g_{\lambda\mu(d)\nu} = |\CYT_{\lambda,d}(\nu)|-|\CYT^-_{\lambda,d-1}(\nu)| =  |\CYT_{\lambda,d}(\nu)|-|\CYT^+_{\lambda,d}(\nu)|
	= |\CYT_{\lambda,d}^-(\nu)|.\qedhere\] 
\end{proof}

\section{Acknowledgments}
The author would like to thank Jonah Blasiak for useful discussions.

\bibliography{hook}
\bibliographystyle{plain}

\end{document}